\documentclass[%referee%
]{amsart}

\usepackage{graphicx, color}
\usepackage{enumerate}
\usepackage{verbatim} % comentarios \begin{comment} Líneas comentadas... \end{comment}

\newcommand{\diam}{\text{diam}}
\newcommand{\dist}{\ensuremath{\text{dist}}}

\newtheorem{theorem}{Theorem}[section]

\newtheorem{example}[theorem]{Example}
\newtheorem{lemma}[theorem]{Lemma}

\newtheorem{corollary}[theorem]{Corollary}

\theoremstyle{definition}
\newtheorem{definition}[theorem]{Definition}

\theoremstyle{remark}
\newtheorem{remark}[theorem]{Remark}

\newtheorem{observation}[theorem]{Observation}

\newcommand{\R}{\mathbb{R}}

\newcommand{\N}{\mathbb N}
\newcommand{\Q}{\mathbb Q}

\begin{document}

\title{Small sets containing any pattern}%

\author[U. Molter and A. Yavicoli]{URSULA MOLTER and ALEXIA YAVICOLI} \thanks{The research for this paper was partially supported by grants UBACyT 2014-2017 20020130100403BA, PIP 11220110101018 (CONICET) and PICT 2014 - 1480, MinCyT}

\address{Departamento de Matem\'{a}tica and IMAS/UBA-CONICET\\
Facultad de Cs. Exactas y
Naturales, Universidad de Buenos Aires \\
Ciudad Universitaria, Pab. I \\
(1428) CABA, Argentina.\\
{\bf Email:} umolter@dm.uba.ar \ ayavicoli@dm.uba.ar}

\begin{abstract}
Given any dimension function $h$, we construct a perfect set $E \subseteq \R$  of zero $h$-Hausdorff measure, that contains any finite polynomial pattern.

This is achieved as a special case of a more general construction in which we have a  family of functions $\mathcal{F}$ that satisfy certain conditions and we construct a perfect set $E$ in $\R^N$, of $h$-Hausdorff measure  zero, such that for any finite set  $\{ f_1,\ldots,f_n\}\subseteq \mathcal{F}$, $E$ satisfies that $\bigcap_{i=1}^n  f^{-1}_i(E)\neq\emptyset$.

We also obtain an analogous result for the images of functions.
Additionally we prove some related results for countable (not necessarily finite) intersections, obtaining, instead of a perfect set, an  $\mathcal{F}_{\sigma}$ set without isolated points.
\end{abstract}

\maketitle

%%%%%%%%%%%%%%%%%%%%%%%%%%%%%%%%%%%%%%%%%%%%%%%%%%%%%%%%%%%%%%%%%%%%%%%%%%
%%%%%%%%%%%%%%%%%%%%%%%%%%%%%%%%%%%%%%%%%%%%%%%%%%%%%%%%%%%%%%%%%%%%%%%%%%
%%%%%%%%%%%%%%%%%%%%%%%%%%%%%%%%%%%%%%%%%%%%%%%%%%%%%%%%%%%%%%%%%%%%%%%%%%
%%%%%%%%%%%%%%%%%%%%%%%%       INTRO      %%%%%%%%%%%%%%%%%%%%%%%%%%%%%%%%
%%%%%%%%%%%%%%%%%%%%%%%%%%%%%%%%%%%%%%%%%%%%%%%%%%%%%%%%%%%%%%%%%%%%%%%%%%
%%%%%%%%%%%%%%%%%%%%%%%%%%%%%%%%%%%%%%%%%%%%%%%%%%%%%%%%%%%%%%%%%%%%%%%%%%

\tableofcontents

\section{Introduction}
\label{intro}
%%%%%%% CONTEXTO DISCRETO %%%%%%%
There is a long history about the study of the relations between the size (in some appropriate sense) of a set and the existence of patterns or prescribed configurations contained in it. In the discrete case, the famous Roth-Szemer\'{e}di theorem \cite{Szemeredi75}, states that any subset of the integers of positive density  contains arithmetic progressions of any length. One can view arithmetic progressions as linear patterns; and there have been many similar results for polynomial patterns, a notable one being the now classical theorem of Furstenberg-S\'ark\"{o}zy \cite{Fur}.

Polynomial patterns with rational coefficients have also been studied in number theory by Bergelson and Leibman \cite{BL96}. They proved that any subset of integers of positive density contains (almost) every polynomial pattern. Later, in \cite{TZ08} and \cite{TZ16} Tao and Ziegler extended the result to the set of prime numbers.

%%%%%%% CONTEXTO CONTINUO %%%%%%%

These type of problems have also been studied in the continuous setting. For example, in \cite{BC16} Boshernitzan and Chaika prove  the following dichotomy for Borel subsets $A \subseteq [0, 1]$: If $H$ stands for the set of homeomorphisms $\phi: [0, 1] \rightarrow [0, 1]$,
either there exists a homeomorphism $\phi \in H$ such that the image $\phi(A)$ contains no 3-term arithmetic progressions;
or, for every $\phi \in H$, the image $\phi(A)$ contains arithmetic progressions of arbitrary finite length.
In fact, they showed that the first alternative holds if and only if the set $A$ is a countable union of nowhere dense sets.

%%%%%%% PATRONES FINITOS (CONTEXTO CONTINUO) %%%%%%%

A set $A$ of real numbers is called universal in measure if every measurable set of positive measure necessarily contains an affine copy of $A$.
From the Lebesgue density theorem, one can deduce that any set $E \subseteq \R$ of positive Lebesgue measure contains similar copies of every finite set (i.e. all finite sets are universal). Iosevich asked the following related question: if $A \subseteq  \R$ is a finite set and $E \subseteq [0,1]$ is a set of Hausdorff dimension $1$, must $E$ contain a similar copy of $A$?
In \cite{Kel} Keleti answered this question by showing that for any set $A \subseteq \R$ of $3$ elements there exists a $1$-dimensional set that contains no similar copy of $A$. Moreover, he proved that given a countable family of triplets, there exists a $1$-dimensional set that does not contain a similar copy of any of those triplets. Maga \cite{Mag} extended Keleti's constructions to the plane.
Using a different method Falconer \cite{Falconer} had previously proved this for only one triangle.
M\'ath\'e \cite{Mat}, and Fraser and Pramanik \cite{FP16} studied similar problems for non-linear patterns, under certain conditions.
In \cite{AY17} the second author obtained finer versions for linear patterns using dimension functions.

In \cite{LP09} {\L}aba and Pramanik gave sufficient conditions to guarantee arithmetic progressions of length $3$ in large sets. They proved that given $E \subseteq \R$ a closed set of Hausdorff dimension $\alpha$, if $\alpha$ is sufficiently close to $1$ and if $E$ supports a probability measure obeying appropriate dimensionality and Fourier decay conditions, then $E$ contains non-trivial $3$-term arithmetic progressions.
Recently, using an harmonic analysis approach, in \cite{HLP16} sufficient conditions where given by the same authors together with Henriot that guarantee that a subset of $\mathbb{R}^N$ contains certain polynomial patterns.

Also assuming hypotheses on the decay of the Fourier transform, in \cite[Corollary 1.7]{CLP16} Chan, {\L}aba, and Pramanik ensure the existence of equilateral triangles contained in a subset of the plane. For recent progress in dimension $4$ and higher, whithout assuming Fourier bounds, see the work of Iosevich and Liu \cite{IL16}.

In this paper we complement these results by explicitly constructing very small perfect sets that {\em do} contain prescribed patterns (in particular, polynomial patterns).

Davies, Marstrand and Taylor \cite{DMT} already proved a result of this kind:
Given a dimension function $h$ (see Section \ref{section def} for the definition), there exists a closed set $E \subseteq \R$ such that $\mathcal{H}^h(E)=0$ and $\bigcap_{i=1}^n (a_iE + b_i) \neq \emptyset$ for any finite subset of linear (real) functions of non-zero slope.
Further there exists an  $\mathcal{F}_{\sigma}$ set $E \subseteq \R$ such that $\bigcap_{i \in \Lambda} (a_i E + b_i) \neq \emptyset$ (for $a_i \neq 0$) for any countable set $\Lambda$ and $\mathcal{H}^h(E)=0$. The first part of this theorem was extended by Keleti, Nagy and Shmerkin to $\R^N$ and the case when the family of functions are affine functions in \cite{KNS}.

We generalize the results from \cite{DMT} and \cite{KNS} in several directions. We allow much more general classes of functions (including polynomials), and consider both images and preimages in the case of non invertible functions.

We prove that given a dimension function $h$ and a family of continuous functions $\mathcal F$ on $\R^N$ satisfying certain conditions (see Theorem \ref{patronesNdim}), then there exists a perfect set $E \subseteq \R^N$, such that $\mathcal{H}^h(E)=0$ and $\bigcap_{1\leq i \leq n} f^{-1}_i(E) \neq \emptyset$ for each finite subset $\{f_1, \cdots, f_n \}$ of $\mathcal F$. In particular the family of non constant polynomials with real coefficients satisfy the assumptions of the previous theorem (see Theorem \ref{patrones polin 1dim} and its generalization Proposition \ref{patrones polin}).

We also prove an analogous result for intersection of images: given a dimension function $h$, and $\mathcal{F}$ a family of continuous functions on $\R^N$ satisfying certain conditions, then there exists a perfect set $E \subseteq \R^N$ such that $\bigcap_{1\leq i \leq n} f_i(E) \neq \emptyset$ for any finite subset $\{f_1, \cdots, f_n \}$ in $\mathcal{F}$, and $\mathcal{H}^h(E)=0$ (see Theorem \ref{teo Ndim cerr}). We show that in fact the conditions on $\mathcal{F}$ are not hard to be satisfied, since for example the set of bilipschitz functions (as well as many others) satisfy them.

%%%%%%% PATRONES INFINITOS (CONTEXTO CONTINUO) %%%%%%%
An old conjecture of Erd\H{o}s \cite{Erd74} states that for any infinite set $A \subseteq  \R$ there exists a set $E \subseteq \R$ of positive Lebesgue measure which does not contain any similar copy of $A$ (Erd\H{o}s' conjecture says that there is no infinite universal set).
In \cite{Fal84} Falconer proved that slowly decaying sequences are not counterexamples. It is not known if $\{2^{-n} : \ n \in \N \}$ is universal.

In this article we investigate a problem wich is in some sense the oposite of Erd\H{o}s': we construct very small $\mathcal{F}_{\sigma}$ sets containing countable predetermined patterns. We prove that given a dimension function $h$ and a family of continuous functions $\mathcal F$ defined on a closed set $D \subseteq \R^N$ satisfying certains conditions, there exists an  $\mathcal{F}_{\sigma}$ set without isolated points $E \subseteq \R^N$, with $\mathcal{H}^h(E)=0$ and $\bigcap_{i \in \Lambda} f^{-1}_i(E) \neq \emptyset$ for any $(f_i)_{i \in \Lambda} \subseteq \mathcal{F}$ countable subset (see Theorem \ref{Fsigma patr Ndim}). We also obtain analogous results for images insted of preimages (see Theorem \ref{Fsigma Ndim}).

The construction of each of these sets is quite subtle and requires a delicate construction. In fact, the way of constructing the required set in different situations relies on similar methods of proof. We included all details only for the first theorem and only indicate the construction for the other cases.

%%%%%%%%%%%%%%%%%%%%%%%%%%%%%%%%%%%%%%%%%%%%%%%%%%%%%%%%%%%%%%%%%%%%%%%%%%
%%%%%%%%%%%%%%%%%%%%%%%%%%%%%%%%%%%%%%%%%%%%%%%%%%%%%%%%%%%%%%%%%%%%%%%%%%
%%%%%%%%%%%%%%%%%%%%%%%%%%%%%%%%%%%%%%%%%%%%%%%%%%%%%%%%%%%%%%%%%%%%%%%%%%
%%%%%%%%%%%%%%%%%%%%%%%%    DEFs NOTs     %%%%%%%%%%%%%%%%%%%%%%%%%%%%%%%%
%%%%%%%%%%%%%%%%%%%%%%%%%%%%%%%%%%%%%%%%%%%%%%%%%%%%%%%%%%%%%%%%%%%%%%%%%%
%%%%%%%%%%%%%%%%%%%%%%%%%%%%%%%%%%%%%%%%%%%%%%%%%%%%%%%%%%%%%%%%%%%%%%%%%%

\section{Definitions and notations}\label{section def}
In this section we will set the notation and give the definitions needed in this paper.

We will call $h$  a {\em dimension function}, if $h:\R_{\geq 0}\to\R \cup\{+\infty\} $, is non-decreasing and  right continuous such that $h(0)=0$ and $h(t) > 0$ for $t>0$.

The set of all dimension functions is partially ordered, considering the order defined by $$h_2 \prec h_1 \text{ if } \lim_{x \to 0^+} \frac{h_1(x)}{h_2(x)} =0.$$

\begin{definition}
If $h$ is a dimension function, the outer Hausdorff measure associated with $h$ is
$$\mathcal{H}^h(E):=\lim_{\delta\to 0} \inf_{} \left\{ \sum_{i}h(|U_i|) : \{U_i\}_{i} \text{ a } \delta
\text{-covering of } E \right\}.$$
\end{definition}
For any dimension function $h$,  $\mathcal{H}^h$ is a Borel measure.
This definition generalises the outer $\alpha$-dimensional Hausdorff measure, which is the particular case  $h(x):=x^{\alpha}$. In this case the parcial order says that $x^s \prec x^t$ if and only if $s<t$.

\begin{definition}\label{defi-pattern}
Let $E \subseteq \R^N$ and  $\mathcal{F}:=\{f_i:\R^N \to \R^N, i \in \Lambda\}$ a set of functions, we say that $E$ contains the pattern $(f_i)_{i \in \Lambda}$ if there exist $t \in \R^N$ such that  $f_i(t) \in E$  $\forall i \in \Lambda$, or equivalently, if
 \begin{equation} \label{eq-intersect} \bigcap_{i \in \Lambda} f^{-1}_i (E) \neq \emptyset.
 \end{equation}
 In the case that
$\Lambda$ is finite, we say that the pattern is finite. If $\Lambda$ es countable, we say that the pattern is countable. If the functions $f_i$ are non-constant polynomials, we say that the pattern is a {\em polynomial pattern}.
\end{definition}

\begin{remark} Our definition of polynomial patterns is more general than all the different definitions of polynomial patterns found in the literature.

Also, it includes arithmetic progressions since they are a special case of polynomial patterns for the case that $f_1 , \cdots, f_n$ are certain similarities.
\end{remark}

Equation \eqref{eq-intersect} is our motivation to study the intersections of the preimages or images of a set under the actions of the functions $f_i$.

%A little word of warning is adequate here:
%\begin{remark} \label{rem-g-delta}
%If $h$ is any dimension function, it is straightforward to prove that there exists a dense $\mathcal{G}_{\delta}$ subset
%$G \subseteq \R^N$, such that $\mathcal{H}^h(G)=0$.
%Let $G = \bigcap_{k \in \N}G_k$ with $G_k$ a dense open subset of $\R^N$.
%Note further that if  for all $i \in \Lambda \subseteq \N$,
%$f_i:\R^N \to \R^N$ are homeomorphisms, we have that
%\begin{equation*}
%\bigcap_{i \in \Lambda} f_i(G)= \bigcap_{i \in \Lambda} \bigcap_{k \in \N} f_i(G_k)
%\end{equation*}
%is the intersection of open dense sets, and because of Baire's category theorem, it is dense in $\R^N$, and hence non-empty.

%So we found a dense $\mathcal{G}_{\delta}$ set
%$G \subseteq \R^N$, such that $\bigcap_{i \in \Lambda} f_i(G)\neq \emptyset$ for any countable family of homeomorphisms and $\mathcal{H}^h(G)=0$.
%\end{remark}

We concentrate on the case of closed sets (or $\mathcal{F}_{\sigma}$ sets), since it is not difficult to find $\mathcal{G}_{\delta}$ sets such that their intersection under countable homeomorphisms is non-empty.

Through the whole paper, bilipschitz and locally bilipschitz functions will play a fundamental role. For convenience of the reader we give the definition here:

\begin{definition} \label{defi-bilip}
A function $f: \R^N \to \R^N$ is {\em bilipschitz} if there exist positive constants  $c_1$ and $c_2$ such that for $x,y \in \R^N$ we have
\begin{equation*}c_1\|x-y\| \leq \| f(x)-f(y) \| \leq c_2 \|x-y \|.\end{equation*}
In the particular case that $c_2 \leq 1$, we say that $f$ is a {\em non-expansive} bilipschitz map. And in the case that
$c_1\geq 1$ we say that it is  a {\em non-contractive} bilipschitz map.
Sometimes, when we want to make explicit reference to the constants we will say that the function is {\em bilipschitz with constants ($c_1, c_2$)}.
\end{definition}

\begin{definition} If $X$ and $Y$ are topological spaces, we say that a function $f: X \to Y$ is a {\em closed map}, if for all closed sets $E$ in $X$, $f(E)$ is closed in $Y$.
\end{definition}

\begin{definition}
A function $\psi: \Omega_{\psi} \subseteq \R^N \to \R^N$ is {\em locally bilipschitz} if for any $x_0 \in \Omega_{\psi}$ there exist $\varepsilon=\varepsilon(x_0)$ and constants  $c_1=c_1(x_0)>0$ and $c_2=c_2(x_0)>0$ such that for $x,y \in B(x_0,\varepsilon)$ we have
\begin{equation*}c_1 \|x-y\| \leq \|\psi(x)-\psi(y)\| \leq c_2 \|x-y\|.\end{equation*}
\end{definition}

Note first that being {\em locally bilipschitz} is indeed a weaker property than being {\em bilipschitz}, since we have the following Observation.

\begin{observation}\label{lemita}
If $\psi: \Omega \to \R^N$ with $\Omega$ a closed subset of $\R^N$ is locally bilipschitz, closed and injective, then:
\begin{enumerate}
\item If $K$ is a compact subset of $\Omega$, $\psi |_K$ is bilipschitz.
\item If $A$ is a compact set in $\R^N$, $K$ is a compact set in $\Omega$ contained in $\psi^{-1}(A)$, then there exists a positive constant $c=c(\psi, A)$ such that
\begin{equation*}\diam( \psi(K)) \geq \frac{ \diam(K)}{c}.\end{equation*}
\end{enumerate}
\end{observation}

The proof uses the compactness of $K$ and the fact that $\psi$ is closed and injective (In fact, neither of these conditions can be removed).

%%%%%%%%%%%%%%%%%%%%%%%%%%%%%%%%%%%%%%%%%%%%%%%%%%%%%%%%%%%%%%%%%%%%%%%%%%
%%%%%%%%%%%%%%%%%%%%%%%%%%%%%%%%%%%%%%%%%%%%%%%%%%%%%%%%%%%%%%%%%%%%%%%%%%
%%%%%%%%%%%%%%%%%%%%%%%%%%%%%%%%%%%%%%%%%%%%%%%%%%%%%%%%%%%%%%%%%%%%%%%%%%
%%%%%%%%%%%%%%%%%%%%%%%%       FINITE     %%%%%%%%%%%%%%%%%%%%%%%%%%%%%%%%
%%%%%%%%%%%%%%%%%%%%%%%%%%%%%%%%%%%%%%%%%%%%%%%%%%%%%%%%%%%%%%%%%%%%%%%%%%
%%%%%%%%%%%%%%%%%%%%%%%%%%%%%%%%%%%%%%%%%%%%%%%%%%%%%%%%%%%%%%%%%%%%%%%%%%

\section{Finite intersections and finite patterns} \label{finite}

In this section we will prove (constructively) the existence of small sets  that have non-empty {\em finite} intersection through some maps or preimages of maps.

\subsection{Small perfect sets with a finite intersection property}
\

\

We will prove that given a dimension function $h$ and a set $\mathcal{F}$ of functions satisfying certain conditions there exists a perfect set $E$  of zero $\mathcal{H}^h$ measure  such that any finite intersection of the images of $E$ by functions of $\mathcal{F}$ is non-empty.

Our tool will be to modify our original functions by composing them with functions $\Psi$ such that the compositions are bilipschitz. This will enable us to look at the intersections of the images of $E$ essentially as an intersection of the images under bilipschitz maps.

%%%%%%%%%%%%%%%%%%%%%%%%%%%%%%%%%
We will need the following lemma to construct the set.

\begin{lemma}\label{lemita2}
Let $h$ be a dimension function and let $L, N \in \N$ be given. There exists a sequence
$(\delta_n)_{n\in \N_0}  \subseteq (0,1]$ satisfying simultaneously:
\begin{enumerate}
\item $\delta_0=1$,
\item $\delta_{n}\leq  \frac{\delta_{n-1}}{4L \sqrt{N}}$,
\item $\lim_{n \to +\infty} \left( \frac{N_1}{\delta_{n-1}}+1 \right)^N h(\delta_n N_2)=0$ for all $N_1, N_2 \in \N$.
\end{enumerate}
\end{lemma}

\begin{proof}
Let $(q_i)_{i \in \N}$ be an enumeration of $\N \times \N$.
We choose first $\delta_0=1$. Once we have chosen $\delta_0, \cdots, \delta_{n-1}$, we take $\delta_n >0$ sufficiently small such that $\delta_n \leq  \frac{\delta_{n-1}}{4L \sqrt{N}}$ and
\begin{equation*}
\left( \frac{\pi_1(q_i)}{\delta_{n-1}}+1 \right)^N h(\delta_n \pi_2(q_i)) <\frac{1}{n} \text { for all } 1 \leq i \leq n,\end{equation*}
 where $\pi_1$ and $\pi_2$ are the projections onto the first and second coordinate respectively.
 Let us see that the constructed sequence satisfies the third condition.
 For, given $N_1 , N_2 \in \N$, there exists $i \in \N$ such that $q_i = (N_1, N_2)$.
 Given $\varepsilon > 0$, if $n > \max\{ i, \frac{1}{\varepsilon} \}$, we have that
 $$\left( \frac{\pi_1(q_i)}{\delta_{n-1}}  +1 \right)^N h(\delta_n \pi_2(q_i))< \varepsilon, $$
 which implies the third statement.
 \end{proof}

\begin{theorem}\label{teo Ndim cerr}
Let $h$ be a dimension function, $\mathcal{F}$ a family of continous functions from $\R^N$ to $\R^N$ such that there exists a countable family of closed, injective, locally bilipschitz functions $\Psi=\{ \psi_r: \Omega_r \to \R^N\}_{r\in \N}$  each defined on a closed set $\Omega_r$ and satisfying:
\begin{itemize}
\item
$\lim_{\| x \| \to +\infty, x \in \Omega_r} \|\psi_r(x)\|=+\infty.$
 \item for each $f_i \in \mathcal{F}$ there exists $\psi_{r(i)} \in \Psi$ such that $f_i \circ \psi_{r(i)}$ is a non-contractive bilipschitz map on $\Omega_{r(i)}$
 \item given any finite number of functions $f_1, \cdots , f_n \in \mathcal{F}$, the set $A_{a, f_1, \cdots, f_n}$ defined as
 \begin{equation*} A_{a, f_1, \cdots, f_n}:=\bigcap_{1 \leq i \leq n} f_i \circ \psi_{r(i)}|_{\Omega_{r(i)}} \left( \left( (-a,a)^N \right)^C \right)
 \end{equation*}
 contains arbitrarily large balls for all $a>0$.
\end{itemize}
Then there exists a perfect set $E \subseteq \R^N$, such that $\mathcal{H}^h(E)=0$ and $\bigcap_{1\leq i \leq n} f_i(E) \neq \emptyset$ for any finite subset $\{f_1, \cdots, f_n \} \subseteq \mathcal{F}$.
\end{theorem}

\begin{proof}

Let $\mathcal{F}=\bigcup_{L \in \N_{\geq 2}} \mathcal{F}_L$, where
\begin{align*}
\mathcal{F}_L:=\{ f\in \mathcal{F}: \ &\exists\ \psi \in \Psi \text{ satisfying all hypothesis} \\
&\text{ such that } f \circ \psi \text{ is bilipschitz with constants  $(1,L)$} \}.
\end{align*}

We will prove that given $L\geq 2$, there exists a closed set
$E_L \subseteq B(0,L)^C$, such that $\mathcal{H}^h(E_L)=0$ and $\bigcap_{1\leq i \leq n}f_i(E_L) \neq \emptyset$ for any finite subset $\{ f_1, \cdots, f_n \}$ of $\mathcal{F}_L$.

To see this, let us fix $L\geq 2$ and let $(\delta_n)_{n\in \N}$ be the sequence given by Lemma~\ref{lemita2}.

Define $\tilde{F}_n \subseteq \R$ as a union of closed intervals of length $\delta_n$  on the positive real line that are equally spaced  and the complementary intervals are of length $\frac{\delta_{n-1}}{L4\sqrt{N}}$. Let $F_n:=\tilde{F}_n^N$ and let $K_j$ be defined as
\begin{equation*}
K_j:=\bigcap_{k \in \N} F_{(2k-1)2^{j-1}}.
\end{equation*}
By the choice of the sequence $(\delta_n)_{n \in \N}$ and the argument below, it follows that $K_j$ is non-empty.

We now fix $m \in \N$. Since by hypothesis $\lim_{\|x\| \to +\infty, x\in \Omega_j}\|\psi_j (x)\|  = +\infty$, for each $j \in \{1, \cdots , m \}$, there exists $k_j \in \N$ such that $\|\psi_j (x)\| \geq \max\{m,L \}$ for all $x \in \left( (-k_j,k_j)^N \right)^C \cap \Omega_j$. We define
\begin{equation} \label{eq-interval}
I_m:=\left( (-\tilde{k}_m,\tilde{k}_m)^N \right)^C \quad \text{with} \quad \tilde{k}_m:=\max_{1\leq j \leq m} k_j,
\end{equation}
and
\begin{equation*}E_L:= \bigcup_{m \in \N} \bigcup_{j=1}^m \psi_j(K_j \cap I_m \cap \Omega_j).
\end{equation*}
Since $\psi_j(K_j \cap I_m \cap \Omega_j) \subseteq \left( B(0,L) \right)^C$ for all $m\in\N$ and all $1 \leq j \leq m$, we have that $E_L \subseteq \left( B(0,L) \right)^C$.

Further, $E_L$ is closed, because $\bigcup_{j=1}^m \psi_j(K_j \cap I_m \cap \Omega_j)$ is contained in $\left( B(0,m) \right)^C$ and is closed since $\psi_j$ is a closed function.

Our sought after {\em closed} set $E$ is
\begin{equation} \label{conjunto}
E:=\bigcup_{L \in \N_{\geq 2}}E_L.
\end{equation}

To be able to proof that $E$ satisfies the desired conditions, we first need to prove the following claims:
\begin{enumerate}
\item {\bf Claim 1.} Given $f_1 , \cdots, f_n \in \mathcal{F}_L$, then $\bigcap_{i=1}^n f_i (E_L) \neq \emptyset$.
\item {\bf Claim 2.} $\mathcal{H}^h(E_L)=0$.
\end{enumerate}

For the proof of Claim 1,
By the hypothesis of the family of functions, to each function
$f_i$ we  associate a function $\psi_{r(i)}$ such that $f_i\circ \psi_{r(i)}$ is bilipschitz with constants $(1,L)$ in $\Omega_{r(i)}$.

Let $R:=\max\{ r(1), \cdots ,r(n) \} \in \N$, and let $I_R$ be given by equation \eqref{eq-interval}.

By hypothesis $A :=\bigcap_{1\leq i \leq n} f_i \circ \psi_{r(i)}|_{\Omega_{r(i)}} (I_R)$ contains arbitrarily large balls.

Hence, using the injectivity of  $f_i\circ \psi_{r(i)}$ we have
\begin{align}\label{cuentita-1}
\bigcap_{i=1}^n f_i(E_L)
&\supseteq \bigcap_{i=1}^n f_i \left( \psi_{r(i)}|_{\Omega_{r(i)}}(K_{r(i)} \cap I_R) \right) \nonumber\\
&\supseteq \left( \bigcap_{i=1}^n f_i \circ \psi_{r(i)}|_{\Omega_{r(i)}} (K_{r(i)}) \right) \bigcap \left( \bigcap_{i=1}^n f_i \circ \psi_{r(i)}|_{\Omega_{r(i)}}(I_R) \right) \nonumber\\
&\supseteq \bigcap_{i=1}^n \bigcap_{k \in \N} f_i \circ \psi_{r(i)}|_{\Omega_{r(i)}}(F_{(2k-1)2^{r(i)-1}}) \cap A.
\end{align}
For each $r(i) \in \{r(1), \dots, r(n)\}$ consider all  those $m \in \N$ that are of the form $(2k-1)2^{r(i)-1}$ with $k \in \N$ and define $g_m:=f_i \circ \psi_{r(i)}|_{\Omega_{r(i)}}$. Let $\Lambda$ be the countable set of those indices which we consider strictly ordered ($m_n< m_{n+1} \  \forall  n \in \N, m_n \in \Lambda$). Note that there may be many $g_m$ that are equal.
Therefore we rewrite equation \eqref{cuentita-1} as
\begin{equation} \label{cuentita}
\bigcap_{i=1}^n f_i(E_L) \supseteq \bigcap_{m\in \Lambda} g_m(F_m) \cap A.
\end{equation}

In order to show that this last intersection  is not empty, we will show that:
\begin{enumerate}
\item There exists  a cube $C_1$ of $F_{m_1}$ such that $C_1 \subseteq g^{-1}_{m_1}(A)$.
\item Given a cube $C_k$ of  $F_{m_k}$ included in $g^{-1}_{m_k} (A)$, there exists a cube $C_{k+1}$ of $F_{m_{k+1}}$ such that $g_{m_{k+1}}(C_{k+1})\subseteq g_{m_k} (C_k)$.
\end{enumerate}
These two conditions yield a sequence of nested compact sets and hence their intersection is not empty.
We show both conditions:
\begin{enumerate}
\item By hypothesis for each $\rho>0$ there exists a ball $B_{\rho}$ of radius $\rho$ inside $A$ wich is contained in $\text{Im}(g_{m_1})$, and since the cubes of $F_{m_1}$ are distributed uniformly throughout the whole $\R^N$ with the same length and separation there must exist a ball of radius $\frac{\rho}{L}$ inside $g^{-1}_{m_1}(B_{\rho}) \subseteq g^{-1}_{m_1}(A)$.

\item Let $x$ be the center of $C_k \subseteq g^{-1}_{m_k}(A)$. We have the following inclusions $\overline{B \left(x, \frac{\delta_{m_k}}{2} \right)} \subseteq C_k \subseteq g^{-1}_{m_k}(A)$ and hence $g_{m_k} \left( \overline{B \left(x, \frac{\delta_{m_k}}{2} \right)} \right) \subseteq g_{m_k}(C_k)$.
Further we can see that
\begin{equation*}
\overline{B \left(g_{m_k}(x), \frac{\delta_{m_k}}{2} \right)} \subseteq g_{m_k} \left( \overline{B \left(x, \frac{\delta_{m_k}}{2} \right)} \right) .
\end{equation*}
For, let $y \in \overline{B \left(g_{m_k}(x), \frac{\delta_{m_k}}{2} \right)}$. If $y \notin g_{m_k} \left( \overline{B \left(x, \frac{\delta_{m_k}}{2} \right)} \right)$, and  since $g_{m_k}$ is injective and continuous, by the Jordan-Brouwer separation theorem, $g_{m_k} \left(\partial \overline{B(x,\frac{\delta_{m_k}}{2})}   \right)$ is the boundary between the regions $g_{m_k} \left( B(x,\frac{\delta_{m_k}}{2}) \right)$ and $\R^N \setminus g_{m_k} \left( \overline{B(x,\frac{\delta_{m_k}}{2})} \right)$. Clearly $g_{m_k}(x)$ is contained in the first region, and by assumption $y$  belongs to the second region. Hence there must exist
$z \in g_{m_k} \left(\partial \overline{B(x,\frac{\delta_{m_k}}{2})}   \right)$ such that $\dist (z, g_{m_k}(x))< \dist (y, g_{m_k}(x))$.

Since $g_{m_k}$ is non-contractively bilipshitz, we have
\begin{equation*}
\frac{\delta_{m_k}}{2} = \dist (g^{-1}_{m_k}(z), x) \leq \dist (z, g_{m_k}(x))< \dist (y, g_{m_k}(x)) \leq \frac{\delta_{m_k}}{2},
\end{equation*}
which is a contradiction.

Combining all inequalities, we have
\begin{equation}\label{inclusion}
\overline{B \left(g_{m_k}(x), \frac{\delta_{m_k}}{2} \right)} \subseteq g_{m_k} \left( \overline{B \left(x, \frac{\delta_{m_k}}{2} \right)} \right) \subseteq g_{m_k}(C_k).
\end{equation}
From here we conclude
\begin{align}\label{conten}
\overline{B}:=\overline{B \left( g^{-1}_{m_{k+1}}(g_{m_k}(x)), \frac{\delta_{m_k}}{2L} \right)} &\subseteq g^{-1}_{m_{k+1}} \left( \overline{B \left( g_{m_k}(x), \frac{\delta_{m_k}}{2L} \right)} \right) \\
&\subseteq g^{-1}_{m_{k+1}}(g_{m_k}(C_k)). \nonumber
\end{align}
where the first inclusion is proven analogously as for equation \eqref{inclusion}.

To see that in fact there is a cube from $F_{m_{k+1}}$ in $\overline{B}$, it is enough to show that
\begin{equation*}
\sqrt{N} \left( \frac{\delta_{m_{k+1}-1}}{4L\sqrt{N}} + \delta_{m_{k+1}} \right) \leq \frac{\delta_{m_k}}{2L}.
\end{equation*}
Since $m_{k+1}>m_k$, we have $\delta_{m_{k+1}-1} \leq \delta_{m_k}$ and $\delta_{m_{k+1}} \leq \frac{\delta_{m_k}}{4L\sqrt{N}}$ by the way we chose the sequence using Lemma~\ref{lemita2}. This proves the desired inequality, and hence Claim 1.
\end{enumerate}

We now turn our attention to {\bf Claim 2}, i.e. to see that  $\mathcal{H}^h(E_L)=0$.

Since $E_L\subseteq \bigcup_{j \in \N} \bigcup_{i \in \N} \psi_j(K_i)$, it suffices to show that
\begin{equation*} \mathcal{H}^h\left(\bigcup_{j \in \N} \bigcup_{i \in \N} \psi_j(K_i)\right)=0.
\end{equation*}
Let $\tilde{I}$ a cube of sidelength $1$, and let  $i, j \in \N$. It will be enough to see that $\mathcal{H}^h(\psi_j(K_i)\cap\tilde{I})=0$.
But since $\psi_j(K_i)\cap\tilde{I} \subseteq \bigcap_{k \in \N}\psi_j(F_{(2k-1)2^{i-1}})\cap\tilde{I}$, let us show that $\mathcal{H}^h_{\delta_n N_2}(\psi_j(F_n)\cap\tilde{I})\longrightarrow_{n \to \infty} 0$ where $N_2:= N_2(\psi_j ,N, \tilde{I}) \in \N$.

Let
\begin{equation*}M:=\#\{ J \text{ cube in } F_n \text{ such that } \psi_j(J)\cap \tilde{I} \neq \emptyset \}.\end{equation*}

Since the diameter of the cubes in $F_n$ is $\sqrt{N} \delta_n \leq \sqrt{N}$, if $J$ is a cube in $F_n$ such that
$J \cap \psi_j^{-1}(\tilde{I}) \neq \emptyset$ we have that $J$ is contained in a compact set $G$:
\begin{equation*}
J \subseteq G:=G(\psi_j, \tilde{I}, N):=\{x: \ \dist(x,\psi_j^{-1}(\tilde{I}))\leq \sqrt{N} \}.
\end{equation*}
Since $\psi_j$ is an injective and closed locally bilipschitz function defined on the closed set $\Omega_j$, when restricted to the compact set $G$, by Lemma~\ref{lemita} we have that $\psi_j$ is bilipschitz with constants say $(a,b)$, where $a$ and $b$ only depend on $N$, $\tilde{I}$ and $\psi_j$. Hence, if $J$ is a cube in $F_n$ such that
$J \cap \psi_j^{-1}(\tilde{I}) \neq \emptyset$, then
\begin{equation*}\diam(\psi_j(J)) \leq b \ \diam(J) \leq b\sqrt{N} \delta_n \leq N_2 \delta_n,\end{equation*}
where $N_2:= N_2(\tilde{I}, \psi_j, N) \in \N$ and so
\begin{equation}\mathcal{H}^h_{\lambda \delta_n} (\psi_j(F_n) \cap \tilde{I}) \leq M h(N_2 \delta_n). \label{hausdorff}
\end{equation}

Let now $Q=Q(\tilde{I},N,\psi_j)$ be a cube with sides parallel to the axes and edges of length $\ell:=\ell(\tilde{I},N, \psi_j)$, that contains $G$. We have
\begin{equation*}
M \leq \#\{ J \text{ cube of } F_n : \ J \subseteq Q \}
\leq  \left\lceil \frac{\ell}{\delta_n + \frac{\delta_{n-1}}{8L\sqrt{N}}} \right\rceil^N
\leq \left( \frac{\ell 8L\sqrt{N}}{\delta_{n-1}} +1 \right)^N,
\end{equation*}
\begin{equation*}
\text{and therefore}\quad M
\leq \left( \frac{N_1}{\delta_{n-1}} +1 \right)^N,
\end{equation*}
where $N_1=N_1(\tilde{I},N,\psi_j) := \lceil \ell 8L\sqrt{N} \rceil \in \N$.
Inserting this into equation \eqref{hausdorff} we have
\begin{equation*}
\mathcal{H}^h_{\lambda \delta_n} (\psi_j(F_n) \cap \tilde{I}) \leq \left( \frac{N_1}{\delta_{n-1}} +1 \right)^N h(N_2 \delta_n)
\end{equation*}
wich tends to $0$ if we choose $(\delta_n)_{n \in \N}$ as in Lemma~\ref{lemita2}.
This completes the proof of Claim 2.

We are now ready to prove that the set $E$ given by \eqref{conjunto}, satisfies the expected thesis of the Theorem. Recall that $E$ is given by:
\begin{equation*}
E:=\bigcup_{L \in \N_{\geq 2}}E_L.
\end{equation*}
We check:
\begin{itemize}
\item $E$ is closed, since for each $L$ we have that $E_L$ is closed and contained in $B(0,L)^C$.
\item $\mathcal{H}^h(E)=0$, since $E$ is a countable union of sets of $\mathcal{H}^h$-measure zero.
\item If $f_1, \cdots, f_n$ is any finite set of functions from $\mathcal{F}$, there exist natural numbers $L_1, \cdots, L_n \in \N_{\geq 2}$ such that $f_i \in \mathcal{F}_{L_i}$. Choosing $L:=\max\{L_1,\cdots,L_n\}$ we have that $\{f_1, \cdots, f_n\}\subseteq \mathcal{F}_L$, and therefore
\begin{equation*}
\bigcap_{i=1}^n f_i(E) \supseteq \bigcap_{i=1}^n f_i(E_L) \neq \emptyset.
\end{equation*}
\end{itemize}
By the first item, $E$ is closed, and the following lemma shows that we can modify the construction of $E$ slightly, in order for it to be perfect and still satisfy all the required conditions.
This completes the proof of Theorem~\ref{teo Ndim cerr}.
\end{proof}

The following lemma shows that if our set has isolates points we can replace it by another set without isolated points, and still satisfying the required properties.

\begin{lemma}\label{sin aislados} Given a dimension function $h$ and $\tilde{E} \subseteq \R^N$ a closed or $\mathcal{F}_{\sigma}$ set such that  $\mathcal{H}^h(\tilde{E})=0$. Then there exists a set $E \supseteq \tilde{E}$ of the same type but without isolated points such that $\mathcal{H}^h(E)=0$.\end{lemma}

\begin{proof}
Let $D$ be the set of (at most countable) isolated points of $\tilde{E}$. For each $x \in D$ choose a ball $B(x, r_x)$ sufficiently small as not to touch neither the set nor any other chosen ball, i.e. $B(x,  r_x) \cap \tilde{E} = \emptyset$ and $B(x, r_x) \cap B(y, r_y) =\emptyset$ for all $x \neq y$ in $D$.

Let now $C$ be a compact set without isolated points, contained in $[0,1]^N$ such that $\mathcal{H}^h(C)=0$.

For each $x \in D$ we can put a rescaled and translated copy $C_x \subseteq B(x, r_x)$ of $C$, wich is compact, has no isolated points and $\mathcal{H}^h(C_x)=0$.
We obtain the sought after set $E$ by considering $E:=\tilde{E} \cup \bigcup_{x \in D} C_x$.
\end{proof}

Clearly, not any set of functions will have the desired property, for example if
$\mathcal{F}:= \{f_1(x):=x^2+1 , \ f_2(x):=-x^2  \}$ for $N=1$. The requirement that the set $A_{a, f_1, \cdots, f_n}$ contains arbitrarily large balls is providing the condition that the images of finite functions in $\mathcal{F}$ have to intersect ``a lot'' at ``infinity''.

\begin{remark}
In general we can not aim to obtain a {\em bounded} set that satisfies the finite intersection property. If the family of functions contains for example all affine functions, there exists no bounded set $E$ such that $\bigcap_{1\leq i \leq n} \varphi_i(E)\neq \emptyset$ for any finite affine functions
 $\varphi_1 , \cdots , \varphi_n$. Since we are looking for sets containing polynomial patterns (in particular linear polynomials), it is reasonable to search for closed sets rather than for compact sets.
\end{remark}

%\subsection{Consecuences of Theorem~\ref{teo Ndim cerr}}\label{aplicNdimcerr}

We will now exhibit some families of functions to which Theorem~\ref{teo Ndim cerr} can be applied.
\begin{corollary}\label{corol bilip N dim}
Let $\mathcal{F}:=\{ f: \R^N \to \R^N \  bilipschitz \}$ and let $h$ be a dimension function. There exists a perfect set $E \subseteq \R^N$, with $\mathcal{H}^h(E)=0$, such that $\bigcap_{i=1}^n f_i(E) \neq \emptyset$ for any finite subset
$\{ f_1, \cdots, f_n \} \subseteq \mathcal{F}$.
\end{corollary}

\begin{proof}
We consider the countable family of linear functions
\begin{equation*}\Psi=\{ \psi: \R^N \to \R^N \ \psi (x)=\lambda x  \ : \ \lambda \in \Q_{>0} \},
\end{equation*}
which  satisfy the hypothesis of Theorem~\ref{teo Ndim cerr}.

For each bilipschitz function $f \in \mathcal{F}$ let $(c,d)$ denote the lower and upper bilipschitz constants.  We associate to $f$ any function
$\psi (x)= \lambda x$ with $\lambda \in \Q_{\geq \frac{1}{c}}$. Any choice yields that $f \circ \psi$ is non-contractively bilipschitz in $\R^N$. Further, since $f \circ \psi: \R^N \to \R^N$ is bilipschitz, by the Domain Invariance Theorem, it is bijective and so $f\circ \psi(\R^N)= \R^N$. In particular, for any $a>0$, we have that $$\left( f\circ \psi \left(  \left( (-a,a)^N \right)^C \right) \right)^C = f \circ \psi \left( (-a,a)^N\right).$$

Hence, given any finite set of functions $f_1, \cdots , f_n \in \mathcal{F}$ and $a>0$ we have
\begin{align*}
A_{a, f_1, \cdots, f_n}&=\bigcap_{1 \leq i \leq n} f_i \circ \psi_{r(i)} \left(  \left( (-a,a)^N \right)^C \right)\\
&= \R^N \setminus \bigcup_{1 \leq i \leq n} f_i \circ \psi_{r(i)} \left( (-a,a)^N \right).
\end{align*}
But $\bigcup_{1 \leq i \leq n} f_i \circ \psi_{r(i)} \left( (-a,a)^N \right)$ is bounded and so $A_{a, f_1, \cdots, f_n}$ contains arbitrarily large balls.
\end{proof}

Note that a consequence of this Corollary  (which could also be obtained from \newline \cite{KNS}) is that there exists a closed set $E$ in $\R^N$ of Hausdorff dimension $0$ such that for any finite set $A$ of $\R^N$, there exists $z_A \in \R^N$ such that $A+z_A \subseteq E$; i.e.~{\em $E$ contains all finite sets of $\R^N$ up to translations.

\begin{comment}
To see this, consider $h$ a dimension function that is 'smaller' than any $x^\alpha$ for every $\alpha > 0$, for example $h(x) = (-\log (x))^{-1}$ for small $x$. By the previous Corollary, there exists a perfect set $E \subseteq \R^N$, with $\mathcal{H}^h(E)=0$ (and so $\dim_H(E)= 0$), such that $\bigcap_{i=1}^n f_i(E) \neq \emptyset$ for any finite subset $\{ f_1, \cdots, f_n \}$ of bilipschitz functions.

Let now $A:=\{x_1, \dots, x_n\}$ be {\em any} finite subset of $\R^N$, and let us consider the functions $f_i(x) = x-x_i$ with $1\leq i \leq n$. Since $f_1, \cdots, f_n$ are bilipschitz functions, we have that
\begin{equation*}(E-x_1) \cap ( E-x_2) \cap \cdots \cap ( E-x_k) \neq \emptyset.\end{equation*}
Taking the point $z_A$ in this intersection, we have that $A +z_A \subseteq E$.

\end{comment}
}

To continue exhibiting sets of functions  that satisfy the conditions of our theorem, we use the notation $\R[x]$ for polynomials in $\R$ with real coefficients.

\begin{example}\label{polin coef posit}
For  the particular case that the dimension of the underlying space is $1$, the family $\mathcal{F}:=\left\{ P: \ P \in \R[x] \text{ is non-constant, with positive principal coefficient} \right\}$ satisfies the hypothesis of Theorem \ref{teo Ndim cerr} .
To see this asociate to each polynomial  $P(x)=\sum_{k=0}^n a_k x^k$ of degree $n$, the function $\psi :[1,+\infty) \to \R$, $\psi(x)=\left(\frac{x}{\tilde{a}}\right)^{\frac{1}{n}}+\tilde{d}$ with $\tilde{a} \in \Q \cap (0,a_n]$ and $\tilde{d}$ sufficiently large such that all the coefficients (except possibly the constant term) of $P \circ \psi$ are positive.

Another example is $\mathcal{F}:=\left\{ P: \ P  \in \R[x]  \text{ is an odd degree polynomial} \right\}.$
\end{example}

\subsection{Small perfect sets containing any finite pattern}
\

\

Since the definition of patterns involves the intersection of preimages, rather than images (see equation \eqref{eq-intersect}), in this section we will prove an analogous result for intersections of preimages.

\begin{theorem}\label{patronesNdim}
Let $h$ be a dimension function, $\mathcal F$ a family of continuous functions from $\R^N$ to $\R^N$ such that there exists
a countable family
$\Psi=\{ \psi_j \}_j$ of injective, closed, continuous functions defined on closed sets $\Omega_j \subseteq \R^N$, to $\R^N$, such that $\psi^{-1}_j$ are locally bilipschitz with ${\displaystyle \lim_{\|x\| \to +\infty, x \in \text{Im}(\psi_j)} \|\psi_j^{-1}(x)\|=+\infty}$ for every $j$, satisfying:
\begin{itemize}
 \item for each $f_i \in \mathcal{F}$ there exists $\psi_{r(i)} \in \Psi$ and a closed set $D_i \subseteq \R^N$ such that $\psi_{r(i)} \circ f_i|_{D_i}$ is well defined and non-expansive  bilipschitz.
  \item for each $a>0$ and any choice of finite functions $f_1, \cdots, f_n \in \mathcal{F}$, we have that
   \begin{equation*}A_{a,f_1, \cdots , f_n}:=\bigcap_{1\leq i \leq n}(\psi_{r(i)} \circ f_i|_{D_i})^{-1}(((-a,a)^C)^N)\end{equation*} contains arbitrarily large balls.
\end{itemize}
Then there exists a perfect set $E \subseteq \R^N$, such that $\bigcap_{1\leq i \leq n} f^{-1}_i(E) \neq \emptyset$ for each finite subset of $\{f_1, \cdots, f_n \} \subseteq \mathcal F$ and  $\mathcal{H}^h(E)=0$.
\end{theorem}

\begin{proof}

The proof is analogous to the proof of Theorem~\ref{teo Ndim cerr}, defining $\mathcal{F}_L$ as \begin{equation*}\left\{ f_i \in \mathcal{F}: \psi_{r(i)} \text{ associated to } f_i \text{ satisfies that } \psi_{r(i)} \circ f_i|_{D_i}
\text{ is bilipschitz } \left(\frac{1}{L},1 \right) \right\},\end{equation*} and considering $\mathcal{F}=\bigcup_{L \in \N_{\geq 2}} \mathcal{F}_L$.

For each $L\geq 2$, we first construct a closed set  $E_L \subseteq B(0,L)^C$ such that $\mathcal{H}^h(E_L)=0$, and $\bigcap_{1\leq i \leq n}f^{-1}_i(E_L) \neq \emptyset$ for any finite subset $\{ f_1, \cdots f_n \}$ of $\mathcal{F}_L$.

For this, for each $j \in \{1, \cdots , m \}$, there is $k_j \in \N$ such that $\|\psi^{-1}_j (x)\| \geq \max\{m,L \}$ for all $x \in \left( (-k_j,k_j)^N \right)^C \cap \text{Im}(\psi_j)$. We define
\begin{equation*}E_L:=\bigcup_{m \in \N} \bigcup_{j=1}^m \psi^{-1}_j(K_j \cap I_m) \subseteq \left( B(0,L) \right)^C.\end{equation*}

Given $f_1, \cdots , f_n$, by hypothesis we have that $A:=\bigcap_{1\leq i \leq n}(\psi_{r(i)} \circ f_i|_{D_i})^{-1}(I_R)$ contains arbitrarily large balls. We also have that
\begin{equation}\label{cuentita5}
\bigcap_{i=1}^n f^{-1}_i(E_L)\supseteq \bigcap_{m\in \mathcal{M}} g^{-1}_m(F_m) \cap A,
\end{equation}
where we define $g_m:=\psi_{r(i)} \circ f_i|_{D_i}$ for those $m$ that are of the form $(2k-1)2^{r(i)-1}$ with $k \in \N$ and any $r(i)$ with  $1\leq i \leq n$. Let us call $\mathcal{M}$ the countable set of these indexes which we will order increasingly.

To prove that the intersection in equation \eqref{cuentita5} is non empty, we see that
\begin{enumerate}
\item There exists a cube $C_1$ of $F_{m_1}$ such that $C_1\subseteq  g_{m_1} (A)$.
\item Given a cube $C_k$ of $F_{m_k}$ that is contained in $g_{m_k} (A)$, there exists a cube $C_{k+1}$ of $F_{m_{k+1}}$ contained in $g_{m_{k+1}} (A)$, such that $g^{-1}_{m_{k+1}}(C_{k+1}) \subseteq g^{-1}_{m_k} (C_k)$.
\end{enumerate}
In this way, since  $g^{-1}_m$ is locally bilipschitz, we have a sequence of nested non  empty compact sets whose intersection will be non-empty.

Further, $\mathcal{H}^h(E_L)=0$.
Since $E_L\subseteq \bigcup_{j \in \N} \bigcup_{i \in \N} \psi^{-1}_j(K_i)$ it is enough to see that if $\tilde{I}$ is a cube of side length $1$, $j, i \in \N$, we have
\begin{equation*} \mathcal{H}^h_{\delta_m \lambda}(\psi^{-1}_j(F_m)\cap\tilde{I})\longrightarrow_{m \to +\infty} 0 \quad \text{ where } \lambda:= \lambda(\psi_j ,N, \tilde{I}).
\end{equation*}

Noting that
\begin{equation*}M:=\#\{ J \text{ cube of } F_m \text{ such that } \psi^{-1}_j(J)\cap \tilde{I} \neq \emptyset \},\end{equation*}
and reasoning like in Claim 2 of Theorem~\ref{teo Ndim cerr} the result follows.

Finally, taking
 \begin{equation*}E:=\bigcup_{L \in \N_{\geq 2}}E_L\end{equation*} and using Lemma \ref{sin aislados}, the theorem follows.
\end{proof}

\begin{comment}
\begin{remark}
Note for example, that if we consider as before in $\R$
 \begin{equation*}\mathcal{F}:= \{f_1:[1,+\infty)\to \R, \ f_1(x):=x\} \cup \{f_2:(-\infty,0]\to\R, \ f_2(x):=x\},
 \end{equation*}
then $f^{-1}_1(\R) \cap f^{-1}_2(\R) = \emptyset$ and we will never be able to find such a set $E$.
Here again, the requirement that the set $A_{a, f_1, \cdots, f_n}$ contains arbitrarily large balls is telling us that the preimages of any finite number of functions of $\mathcal F$ have to have {\em large} intersection at infinity. This will enable us to build a sequence of nested non-empty compact sets inside of the finite intersections of the preimages guaranteeing a nonempty intersection.
\end{remark}
\end{comment}

  %Note that that the hypothesis are similar to the ones of Theorem~\ref{teo Ndim cerr}, and therefore analogous remarks are in place.
  %The intersection of compact sets is in the domain of the compositions and each cube we choose has to lay in the image of the corresponding composition of functions.

%%%%%%%%%%%%%%%%%%%%%%%%%%%%%%%%%%%%%%%%%%%%%%%%%%%%%%%%%%%%%%%%%%%%

\

At this stage, we are ready to prove the result which we were looking for: to construct a set $E$ of zero $H^h$-measure that contains every finite polynomial pattern. We obtain this result for the one variable case as a
particular case of Theorem \ref{patronesNdim}. Then we extend it to polynomials of several variables in Theorem~\ref{patrones polin}.

\begin{theorem}\label{patrones polin 1dim}
Let $h$ be a dimension function, $\mathcal{P}$ the family of non-constant polynomials in one variable with real coefficients. Then there exists a perfect set $E \subseteq \R$, such that $\mathcal{H}^h(E)=0$ and $\bigcap_{1\leq i \leq n} P^{-1}_i(E) \neq \emptyset$ for any finite subset
$\{P_1, \cdots, P_n \}$ in $\mathcal{P}$. In particular $E$ contains any finite polynomial pattern.

If $h$ is taken adequately, $E$ is a perfect set that has Hausdorff dimension zero and contains any finite polynomial pattern.
\end{theorem}

\begin{proof}
We consider the family of closed, injective, continuous functions
% from $\Omega_{\psi}:=[0,+\infty)$ or $(-\infty, 0]$, to $\R$ ,
\begin{align*}
\Psi &:=\{ \psi: [0,+\infty)\to \R, \ \  \psi(x):=q x^\frac{1}{n} : \ n \in \N , \ q \in \Q_{>0} \}\\
 & \cup \{ \psi: (-\infty, 0] \to \R, \ \ \psi(x):= q (-x)^\frac{1}{n} : \ n \in \N, \ q \in \Q_{>0} \};\end{align*}
such that for each $\psi \in \Psi$, we have
\begin{itemize}
\item $\psi^{-1}$ is injective, closed and locally bilipschitz and
\item $\lim_{x \to +\infty} |\psi^{-1}(x)|=+\infty$.
\end{itemize}
This family verifies that for each $P(x):=\sum_{k=0}^n a_k x^k \in \mathcal{P}$ with $a_n \neq 0$
($n \geq 1$), we can choose $\psi \in \Psi$ as
\begin{equation*}
\psi(x) = \begin{cases}  q x^\frac{1}{n} \text{ with } q \in \Q \cap \left[ \frac{1}{2 |a_n|^{\frac{1}{n}}}, \frac{3}{4 |a_n|^{\frac{1}{n}}} \right] & \text{if $a_n > 0$} \\
q (-x)^\frac{1}{n} \text{ with } q \in \Q \cap \left[ \frac{1}{2 |a_n|^{\frac{1}{n}}}, \frac{3}{4 |a_n|^{\frac{1}{n}}} \right] & \text{if $a_n < 0$}.
\end{cases}
\end{equation*}
We choose $M_P \in \N$ such that
 \begin{align} & |P| \text{ is injective in } [M_P-1,+\infty) \text{ and we have} \label{M1}\\
& \frac{1}{4} \leq  |(\psi \circ P|_{[M_P-1, +\infty)})'(x) | \leq 1 \text{ for all } x > M_P-1. \label{M2}
 \end{align}
 In other words, $P|_{[M_P-1,+\infty)}$ is always positive or always negative, injective and $\psi \circ P|_{[M_P-1, +\infty)}$ is well defined, i.e. $\text{Im}(P|_{[M_P-1,+\infty)}) \subseteq \text{Dom}(\psi)$.

Note that we can require the condition on the derivative, since
\begin{equation*}|(\psi \circ P)'(x) |=q \frac{| \pm \sum_{k=0}^{n-1} a_{k+1} \frac{k+1}{n} x^k |}{|(\pm \sum_{k=0}^n a_k x^k)^{1-\frac{1}{n}}|} \longrightarrow_{x \to +\infty} q |a_n|^{\frac{1}{n}} \in \left[ \frac{1}{2}, \frac{3}{4} \right].\end{equation*}
In particular, $\psi \circ P$ is non-expansive bilipschitz with constants $(\frac{1}{4},1)$ in $[M_P,+\infty)$.

Since $\psi \circ P|_{[M_P, +\infty)}$ is injective, we can define its inverse.

Further $\lim_{x \to +\infty} (\psi \circ P)^{-1} (x)=+\infty$.
Moreover
\begin{equation*}\psi \circ P: (M_P-1, +\infty) \to \psi \circ P(M_P-1, +\infty)=\left(\frac{1}{2} |P(M_P-1)|^{\frac{1}{n}}, +\infty\right)\end{equation*}
is open, and so $(\psi \circ P)^{-1}: \left[\frac{1}{2} |P(M_P)|^{\frac{1}{n}}, +\infty\right) \to [M_P, +\infty)$
%\begin{equation*}(\psi \circ P)^{-1}: \psi \circ P[M, +\infty)=\left[\frac{1}{2} |P(M)|^{\frac{1}{n}}, +\infty\right) \to [M, +\infty)\end{equation*}
is continuous.
Finally, since $\psi \circ P|_{[M_P, +\infty)}$ is increasing, $(\psi \circ P|_{[M_P, +\infty)})^{-1}$ is increasing as well.

By Theorem~\ref{patronesNdim} with $N=1$ and $\mathcal{F}=\mathcal{P}$, associating to each $P$ a function $\psi$ and $M_P$ as indicated above, we obtain the desired result.
\end{proof}

In fact, we can extend this result to the case of polynomials in several variables.
\begin{theorem}\label{patrones polin}
Given a dimension function $h$, and given $\mathcal{\tilde{P}}$ the family of non constant polynomials in several variables $\tilde{P}: \R^N \to \R$, then there exists a perfect set $E \subseteq \R$ such that $\mathcal{H}^h(E)=0$, and $\bigcap_{1\leq i \leq n} \tilde{P}^{-1}_i (E)\neq \emptyset$ any finite subset
$\{\tilde{P_1}, \cdots, \tilde{P_n} \}$ in $\mathcal{\tilde{P}}$.
\end{theorem}

 \begin{proof}
We choose $E$ the set given by Theorem \ref{patrones polin 1dim}.
Given $\tilde{P}_1 , \cdots , \tilde{P}_n \in \mathcal{\tilde{P}}$, it will be enough to choose $\lambda_2, \cdots, \lambda_N \in \R$ such that \\
$P_1(t):=\tilde{P_1}(t,\lambda_2 t, \cdots, \lambda_N t), \cdots, P_n(t):=\tilde{P_n}(t,\lambda_2 t, \cdots, \lambda_N t)$ are non constant polynomials in one variable, because by Theorem \ref{patrones polin 1dim} there exists $t \in \R$ such that
\begin{align*}
  &\tilde{P_1}(t,\lambda_2 t, \cdots, \lambda_N t) \in E \\
  &\cdots  \\
  & \tilde{P_n}(t,\lambda_2 t, \cdots, \lambda_N t) \in E
\end{align*}
from which the result will follow.

To choose $\lambda_2, \cdots, \lambda_N$, let $d_k := \text{degree}(\tilde{P_k})$ and we write
\begin{equation*}\tilde{P_k}(x_1, \cdots, x_N)= \sum_{0 \leq j \leq d_k} \sum_{i_1 +\cdots + i_N=j} a^{(k)}_{i_1, \cdots, i_N} x^{i_1}_1 \cdots x^{i_N}_N,
\end{equation*}
where $a^{(k)}_{i_1, \cdots, i_N} \neq 0$ for some $i_1 +\cdots + i_N= d_k$.

Hence
\begin{equation*}P_k(t):=\tilde{P_k}(t,\lambda_2 t, \cdots, \lambda_N t)= \sum_{0 \leq j \leq d_k} t^j \sum_{i_1 +\cdots + i_N=j} a^{(k)}_{i_1, \cdots, i_N} {\lambda_2}^{i_2} \cdots {\lambda_N}^{i_N},\end{equation*}
where $a^{(k)}_{i_1, \cdots, i_N} \neq \emptyset$ for some $i_1 +\cdots + i_N= d_k$.

Since there always exist $\lambda_2, \cdots, \lambda_N \in \R$ satisfying
\begin{equation*}\sum_{i_1 +\cdots + i_N=d_k} a^{(k)}_{i_1, \cdots, i_N} {\lambda_2}^{i_2} \cdots {\lambda_N}^{i_N}\neq 0 \ \forall 1\leq k \leq n,\end{equation*}
they yield the desired construction.
\end{proof}

\begin{remark}
 We were unable to obtain a similar result for {\em vector valued polynomial patterns} using our method of proof. This is due to the fact that we would
need that if $P$ is a vector valued polynomial from $\R^N \rightarrow \R^N$, i.e. $P:=(P_1, \cdots, P_N)$ with $P_j:\R^N \to \R$ and $N\geq 2$, there exists a set $D_P \subseteq \R^N$ that contains arbitrarily large balls, such that $P|_{D_P}$ is injective.

It is not straightforward to characterize a family of vector valued polynomials that would satisfy such a condition, but it is easy to see that not {\em any} family of polynomials will satisfy it, since for example in $\R^2$, the polynomial function $P(x,y)=(x-y, (x-y)^2)$ will never satisfy such a condition.
\end{remark}

%%%%%%%%%%%%%%%%%%%%%%%%%%%%%%%%%%%%%%%%%%%%%%%%%%%%%%%%%%%%%%%%%%%%%%%%%%
%%%%%%%%%%%%%%%%%%%%%%%%%%%%%%%%%%%%%%%%%%%%%%%%%%%%%%%%%%%%%%%%%%%%%%%%%%
%%%%%%%%%%%%%%%%%%%%%%%%%%%%%%%%%%%%%%%%%%%%%%%%%%%%%%%%%%%%%%%%%%%%%%%%%%
%%%%%%%%%%%%%%%%%%%%%%%%      INFINITE    %%%%%%%%%%%%%%%%%%%%%%%%%%%%%%%%
%%%%%%%%%%%%%%%%%%%%%%%%%%%%%%%%%%%%%%%%%%%%%%%%%%%%%%%%%%%%%%%%%%%%%%%%%%
%%%%%%%%%%%%%%%%%%%%%%%%%%%%%%%%%%%%%%%%%%%%%%%%%%%%%%%%%%%%%%%%%%%%%%%%%%

\section{Infinite intersections and infinite patterns}

We now turn our attention to countable intersections, rather than finite ones.

\subsection{Small $\mathcal{F}_{\sigma}$ sets with a countable intersection property}
\

\

Given a set $\mathcal{F}$ of functions satisfying certain conditions, we will prove that there exists a small $\mathcal{F}_{\sigma}$ set whithout isolated points such that the images under countable intersections of functions of $\mathcal{F}$ is non-empty.

\begin{theorem}\label{Fsigma Ndim}
Let $h$ be a dimension function,  $\mathcal{F}$ a family of continuous functions from $\R^N$ to $\R^N$
such that there exists a sequence of closed and locally bilipschitz functions $\Psi:= (\psi_j)_j$ defined on closed sets $\Omega_j \subseteq \R^N$, satisfying that there exists $L \in \N_{\geq 2}$ such that
\begin{itemize}
\item for a given $f_i \in \mathcal F$, there exists $\psi_{r(i)} \in \Psi$ such that $f_i \circ \psi_{r(i)}|_{\Omega_{r(i)}}$ is well defined and bilipschitz with constants $(1,L)$;
\item given a countable number of functions $(f_i)_{i \in \Lambda \subseteq \N} \subseteq \mathcal{F}$ the set
\begin{equation*}
A_{\{f_i: \ i \in \Lambda\}}:=\bigcap_{i \in \Lambda} f_i \circ \psi_{r(i)}(\Omega_{r(i)}),
\end{equation*}
contains arbitrarily large balls.
\end{itemize}
Then there exists an $\mathcal{F}_{\sigma}$ set without isolated points $E \subseteq \R^N$, with $\mathcal{H}^h(E)=0$, such that $\bigcap_{i \in \Lambda} f_i(E) \neq \emptyset$ for any countable family $\{f_i \ : i \in \Lambda \} \subseteq \mathcal{F}$.
\end{theorem}

\begin{proof}
The proof of this theorem, is analogous to the one of Theorem~\ref{teo Ndim cerr}. We first start constructing an $\mathcal{F}_{\sigma}$ set $E \subseteq \R^N$, such that $\bigcap_{i\in \Lambda} f_i (E) \neq \emptyset$ for any countable subset $(f_i)_{i \in \Lambda \subseteq \N} \subseteq \mathcal{F}$.

Consider
\begin{equation*}E:=\bigcup_j \bigcup_i  \psi_j (K_i) \ \in \mathcal{F}_{\sigma}.\end{equation*}

By hypothesis, to each $f_i$ we have a corresponding $\psi_{r(i)}$ such that $f_i \circ \psi_{r(i)}$ is injective in $\Omega_{r(i)}$. As before we have
\begin{equation}\label{cuentitab}
\bigcap_{i \in \Lambda} f_i(E) \supseteq  \bigcap_{m \in \mathcal{M}} g_m (F_m),
\end{equation}
where again $g_m:=f_i \circ \psi_{r(i)}|_{\Omega_{r(i)}}$ for those indexes $m$ such that  $m=(2k-1)2^{i-1}$ with $k \in \N$ and $i \in \Lambda$ and we denote by $\mathcal{M}$ the countable set of those indexes ordered increasingly.

Since by hypothesis $A:=A_{\{f_i: \ i \in \Lambda \}}:=\bigcap_{i \in \Lambda } f_i \circ \psi_{r(i)}(\Omega_{r(i)})$ contains arbitrarily large balls, $\text{Im}(g_m)$ contains arbitrarily large balls for $m \in \mathcal{M}$.

To see that the intersection in equation \eqref{cuentitab} is non-empty, we see that
\begin{enumerate}
\item There exists a cube $C_1$ of $F_{m_1}$ such that $C_1\subseteq g^{-1}_{m_1}(A)$.
\item Given a cube $C_n$ of $F_{m_n}$ contained in $g^{-1}_{m_n} (A)$, there exists a cube
$C_{n+1}$ of $F_{m_{n+1}}$ contained in $g^{-1}_{m_{n+1}} (A)$ such that $g_{m_{n+1}}(C_{n+1})\subseteq g_{m_n} (C_n)$.
\end{enumerate}
Since $g_m$ is bilipschitz, we constructed a sequence of non-empty nested compact sets whose intersection is non-empty.

Finally note, that by an argument similar to the one used in the proof of Theorem~\ref{teo Ndim cerr}, we can show that $\mathcal{H}^h(E)=0$.
By Lemma \ref{sin aislados} the set $E$ can be taken without isolated points.
\end{proof}

%Note that again the same observations as in Remak~\ref{rem-4-2} are relevant.

%\subsection{Applications of Theorem \ref{Fsigma Ndim}}

One application of the Theorem \ref{Fsigma Ndim} is again the fact that one can obtain a set of zero dimensional Hausdorff measure, such that it contains any countable set of $\R^N$ {\em up to translations}. Precisely, looking at $\mathcal{F}$ as the set of all translations, $\Psi$ as the set of all rational translations and $\Omega_j = \R^N$ for all $j$ we can construct an $\mathcal{F}_{\sigma}$ set without isolated points $E \subseteq \R^N$ of Hausdorff dimension zero, such that for any countable set $A$ of $\R^N$, there exists $z_A \in \R^N$ such that $A+z_A \subseteq E$.
%\begin{comment}
%As for the finite case,  let $h$ a dimension function that is 'smaller' than any $x^\alpha, \alpha > 0$ (for example $h(x) = (-\log (x))^{-1}$ for small $x$) . For any countable subset of $\R^N, \{x_j\}_{j\in \Lambda}$  we consider the functions $f_j(x) = x-x_j, j \in \Lambda$. Since $f_j \in \mathcal{F}$ there exists $\tilde E\subseteq \R^N$ such that $\mathcal{H}^h(\tilde E) = 0$ (and so $\dim(\tilde E)= 0$) and such that
%\begin{equation*}\bigcap_{j \in \Lambda} (\tilde E-x_j) \neq \emptyset.\end{equation*}
%Taking $z = z(\{x_j\}_{j \in \Lambda})$ in this intersection, $\{x_j\}_{j \in \Lambda} +z \subseteq E$.
%\end{comment}

\begin{corollary}
Given  $L \in \N_{\geq 2}$, the following family satisfies the hypotesis of Theorem \ref{Fsigma Ndim}
 \begin{align*}
 \mathcal{F}_L:=\{ f: \R^N \to \R^N & : \exists \ a>0, \ \exists \  b\in (0,La) \text{ such that }\\
 & \ \ a\|x-y\| \leq \|f(x)-f(y)\| \leq b\|x-y\| \ \forall x,y \}.\end{align*}
\end{corollary}

\begin{proof}
Let $\Psi := \{ \psi: \R^N \to \R^N : \psi(x):=\lambda x \text{ con } \lambda \in \Q_{>0} \}$ be a countable family of closed, locally bilipschitz functions defined on $\R^N$, and let $\{f_r\}_{r\in \Lambda}$ a countable family in $\mathcal F$. For each $f_r$ there exist
$c_r, d_r >0$ such that $d_r < L c_r$ and
 \begin{equation*}c_r \|x-y\| \leq \|f_r (x)-f_r (y)\| \leq d_r \|x-y\| \ \forall x, y \in \R^N.\end{equation*}
Taking
$\psi_{r}(x) :=\lambda_{r}x$ with $\lambda_{r} \in \Q \cap [\frac{1}{c_r}, \frac{L}{d_r}]$, then $\psi_r \in \Psi$ and
\begin{equation}
\begin{array}{lll}
\|x-y\| &\leq c_r \|\psi_{r}(x-y)\| \leq \|f_r \circ \psi_{r}(x)- f_r \circ \psi_{r}(y)\| \\
&\leq d_r \|\psi_{r}(x-y)\| \leq d_r \frac{L}{d_r}\|x-y\| =L \|x-y\|.
\end{array}
\end{equation}
Further, since $f_r \circ \psi_r: \R^N \to \R^N$ is bilipschitz, it is bijective and therefore $f_r \circ \psi_r (\R^N)= \R^N$ and hence contains arbitrarily large balls, and therefore the hypothesis of Theorem~\ref{Fsigma Ndim} are satisfied.
\end{proof}

\begin{corollary}
The family $\mathcal F$ of all invertible affine transformations from $\R^N$ to $\R^N$, i.e.
\begin{equation*}\mathcal{F}:=\{ f: \R^N \to \R^N  \ f(x):=Ax+b: \ A \in \R^{N\times N} \text{ invertible, } b\in \R^N \}.\end{equation*}
satisfies the hypothesis of Theorem \ref{Fsigma Ndim}.
\end{corollary}

\begin{proof}
Consider the countable family
\begin{equation*}\Psi:=\{ \psi:\R^N \to \R^N \ \psi(x):=Cx \text{ with } C \in \Q^{N\times N}\}=(\psi_j)_{j \in \N}.\end{equation*}

For each $f_i(x) = Ax+b \in \mathcal{F}$, since $A$ is invertible, by the density of the matrices with rational entries we can find $C\in \Q^{N \times N}$ such that $\| AC-\frac{3}{2}I \| < \frac{1}{2}$. Therefore
\begin{equation*}\| x \| \leq \|AC x \| \leq 2 \| x\| \text{ for all }x \in \R^N.\end{equation*}
Associating to each $f_i$ the function $\psi_{r(i)}(x) = Cx$, we have that $f_i \circ \psi_{r(i)}$ is bilipschitz with constants  $(1,2)$ in all of $\R^N$.

Further, $\bigcap_{i \in \Lambda}f_i \circ \psi_{r(i)}(\R^N)=\R^N$. The result follows from Theorem~\ref{Fsigma Ndim}.
\end{proof}

\subsection{Small $\mathcal{F}_{\sigma}$ sets containing countable any pattern}
\

\

In this section we will concentrate on preimages rather than images. We will show that given a set $\mathcal{F}$ of functions satisfying certain conditions, there exists a small $\mathcal{F}_{\sigma}$ set whithout isolated points such that the preimages under countable intersections of functions of $\mathcal{F}$ is non-empty.

%In this section we will prove the existence of an $\mathcal{F}_{\sigma}$ set without isolated points $E \subseteq \R^N$, such that  $\mathcal{H}^h(E)=0$ and $E$ contains certain patterns.

\begin{theorem}\label{Fsigma patr Ndim}
Let $h$ be a dimension function. Let $\mathcal{F}$ be a family of continuous functions defined on a closed set $D \subseteq \R^N$ containing arbitrarily large balls, such that there exist $L \in \N_{\geq 2}$ and a countable family $\Psi:=\{ \psi_j \}_{j \in \N}$ of continuous, injective and closed  functions defined on closed sets $\Omega_j \subseteq \R^N$, such that $\psi_j^{-1}$ are locally  bilipschitz, and for each $f \in \mathcal{F}$, there exists $\psi_j \in \Psi$ such that $\psi_j \circ f$ is well defined and is bilipschitz of constants  $(\frac{1}{L},1)$ on $D$.

Then there exists an  $\mathcal{F}_{\sigma}$ set without isolated points $E \subseteq \R^N$, with $\mathcal{H}^h(E)=0$, such that $\bigcap_{i \in \Lambda} f^{-1}_i(E) \neq \emptyset$ for any $(f_i)_{i \in \Lambda} \subseteq \mathcal{F}$ countable subset. In other words,  $E$ contains {\em any} countable pattern of $\mathcal{F}$.
\end{theorem}

\begin{comment}
Note that the hypothesis are similar to the ones of Theorem~\ref{Fsigma Ndim}. The set $D$ takes the place of the set $A_{\{f_i : i \in  \Lambda\}}$. Since we want to prove the existence of any pattern in the set $E$, $D$ has to be in the domain of {\em all} functions in $\mathcal F$. The nested sets will be in the domain and each cube that we use to construct the sequence of nested compact sets will be in the image of the respective composition of functions.
\end{comment}

\begin{proof}

Again we follow the same scheme of proof.

We first show that there exists an $\mathcal{F}_{\sigma}$ set $E \subseteq \R^N$, such that
$\bigcap_{i \in \Lambda} f^{-1}_i(E) \neq \emptyset$ for any $(f_i)_{i \in \Lambda} \subseteq \mathcal{F}$ and then show that this set satisfies that  $\mathcal{H}^h(E) = 0.$

For this, consider
\begin{equation*}E:= \bigcup_j \bigcup_k  \psi^{-1}_j(K_k),\end{equation*}
which, by the hypothesis on $\Psi$, is $\mathcal{F}_{\sigma}$.
For $i \in \Lambda$ we have $j(i)$ such that $\psi_{j(i)} \circ f_i$ is non-expansive bilipschitz on $D$.
We have
\begin{equation}\label{inters patr 1 dim}
\bigcap_{i \in \Lambda} f^{-1}_i (E) \supseteq \bigcap_{m \in \mathcal{M}} g^{-1}_m (F_m),
\end{equation}
where (as before) $g_m:=\psi_{j(i)} \circ f_i$ if $m=(2k-1)2^{i-1}$ with $i \in \Lambda$ and $k \in \N$ and $\mathcal{M}:=(m_n)_{n \in \N}$ is the set of those indexes, ordered increasingly.

To see that the intersection in equation \eqref{inters patr 1 dim} is non empty, we argue as before noting that
\begin{itemize}
\item [A)] there exists a cube  $C_1$ of $F_{m_1}$, contained in $\text{Im}(g_{m_1})=g_{m_1}(D)$.
\item [B)] given a cube $C_n$ of $F_{m_n}$, contained in $\text{Im}(g_{m_n})=g_{m_n}(D)$; there exists a cube $C_{n+1}$ of $F_{m_{n+1}}$ which is contained in $g_{m_{n+1}}(D)$ such that
\begin{equation*}g^{-1}_{m_{n+1}}(C_{n+1}) \subseteq g^{-1}_{m_n} (C_n)\end{equation*}
\end{itemize}
We therefore constructed a sequence of compact nested sets (since $g_m$ is bilipschitz) and therefore
the intersection is non-empty.

Finally, analogously as in the proof of Theorem~\ref{patronesNdim} the set $E$  satisfies that $\mathcal{H}^h (E)=0$.
By Lemma \ref{sin aislados} the set $E$ can be taken without isolated points.
\end{proof}

\begin{remark}
We were unable to obtain a result for a countable number of polynomials even in $\R$ using our technique of proof of Theorem~\ref{Fsigma patr Ndim}.

This comes from the fact that if we want a result about countable intersections, using the cited theorem, we would need to find a domain $D$ such that for each polynomial $P$ there exists a function $\psi$ such that $\psi \circ P$ is bilipschitz on $D$. Hence, $P$ would have to be injective on $D$, independently of the polynomial $P$, which is clearly impossible.
\end{remark}

\section*{Acknowledgements}

We thank the anonymous referees for their comments that helped to improve the presentation of this paper, in particular pointing out an observation wich simplified our earlier argument (Lemma \ref{sin aislados}).

%%%%%%%%%%%%%%%%%%%%%%%%%%%%%%%%%%%%%%%%%%%%%%%%%%%%%%%%%%%%%%%%%%%%%%%%%%
%%%%%%%%%%%%%%%%%%%%      THE BIBLIOGRAPHY    %%%%%%%%%%%%%%%%%%%%%%%%%%%%
%%%%%%%%%%%%%%%%%%%%%%%%%%%%%%%%%%%%%%%%%%%%%%%%%%%%%%%%%%%%%%%

%%%%%%%%%%%%%%%%%%%%%%%%%%%%%%%%%%%%%%%%%%%%%%%%%%
%%%%%%%%%%%%%%%%%% Appendix %%%%%%%%%%%%%%%%%%%%%%%%%%%


\begin{thebibliography}{99}
\bibitem{Kel}
Keleti, T., Construction of one-dimensional subsets of the reals not containing similar copies of given patterns.
  Analysis and PDE
\end{thebibliography}
\end{document}